\documentclass[amscd, amssymb, verbatim, 12pt]{amsart}
\usepackage[margin=1in]{geometry}                
\usepackage{graphicx}
\usepackage{amssymb}
\usepackage{epstopdf}
\usepackage{dsfont}
\usepackage{amsmath}
\usepackage{mathrsfs}
\usepackage{natbib}
\usepackage[foot]{amsaddr}

\newtheorem{Theorem}{Theorem}
\newtheorem{Lemma}{Lemma}

\title[Loewy structure of PIMS for $A_{10}$ in char $3$]{The Loewy structure of the projective indecomposable modules for $A_{10}$ in characteristic $3$}
\begin{document}
\author{Stuart Martin and Ha Thu Nguyen$^1$} 
\address{DPMMS, CMS, Wilberforce Road, Cambridge, CB3 0WB, England UK.}

\email{S. Martin@dpmms.cam.ac.uk}
\address{$^1$Corresponding author}

\email{H.T.Nguyen@dpmms.cam.ac.uk}

\begin{abstract}
We compute the Loewy structure of the indecomposable projective modules for the group algebra $FG$, where G is the alternating group on $10$ letters and F is an algebraically closed field of characteristic $3$.
\end{abstract}
\maketitle
\section{Introduction}\label{intro}
Throughout, $F$ is an algebraically closed field of characteristic $3$.  We denote each simple module for a group by its dimension, together with a subscript if there is more than one simple module of that dimension.  Thus, the simple $FA_{10}$ modules  are denoted $1, 34, 41, 84, 224$ (lying in the principal block $B_0$ of defect $4$), $9, 36, 90, 126, 279$ (in a block $B_1$ of defect $2$), and $567$ (a projective simple in the unique block $B_2$ of defect $0$). Since blocks of defect $0$ are easy to describe, we shall only be interested in the blocks of non-zero defect. We denote the central idempotents for the principal block and the block of defect $2$ by $e_0$ and $e_1$ respectively. The central idempotent for the principal block of $FA_9$ is denoted $f_0$.

Our methodology employs standard techniques such as Frobenius Reciprocity together with computation using the MeatAxe. The MeatAxe is used to
find full or partial submodule lattices of certain induced modules, and we thus obtain sufficient information to determine the Loewy
structures of the indecomposable projectives for  $FA_{10}.$ As a by-product of this process, we also obtain information about all the spaces Ext$^{1}_{A_{10}}(S,T).$ This can be viewed as the first step in constructing the Ext-quiver for this group. To our knowledge, no  ``reasonably nice" Ext-quivers of the
alternating groups are in the published literature - one reason for this may be the difficulty of generating a consistent labelling for the simple modules.

The main results are now stated.

\begin{Theorem}\label{T:Main1}
The Loewy structure of the projective indecomposable modules of $B_0$ are as follows:

\begin{center}
$1$ 

$34 \quad 41 \quad84 \quad 224$

$1 \quad 1 \quad 1  \quad 1 \quad 34 \quad 41 \quad 41 \quad 84 \quad 224 \quad 224$

$1 \quad 1 \quad 1  \quad 34 \quad 34 \quad 34  \quad 41 \quad 41  \quad 41 \quad 84 \quad 84\quad 224 \quad 224$

$1 \quad 1 \quad 1  \quad 1 \quad 34 \quad 41  \quad 41 \quad 84 \quad 224 \quad 224$

$34 \quad41  \quad84 \quad 224$
 
 $1$

 \vspace{20pt}
 
\parbox{3in}{
\begin{center}
$34$ 

$1 \quad 41 \quad 224$

$1 \quad 34 \quad 34  \quad 41  \quad 84 \quad 224 $

$1 \quad 1  \quad 1 \quad 34 \quad 41  \quad 41  \quad 84 \quad 224 $

$1 \quad 34 \quad 34  \quad 41  \quad 84 \quad 224 $

$1 \quad 41 \quad 224$

$34$
\end{center}}
 \parbox{3in}{
\begin{center} 
$41$ 

$1 \quad 34 \quad 84 \quad 224$

$1 \quad1 \quad 34 \quad 41 \quad 41  \quad 41  \quad 84 \quad 224 $

$1 \quad 1  \quad 1 \quad 34 \quad 34 \quad 41  \quad 84   \quad 84 \quad 224 \quad 224 $

$1 \quad 1 \quad 34 \quad 41 \quad 41   \quad41 \quad 84 \quad 224 $

$ 1 \quad 34 \quad 84 \quad 224$

$41$ 
\end{center}}

\vspace{20pt}

 \parbox{3in}{
\begin{center} 
 $84$ 

$1 \quad 41 \quad 84 $

$1 \quad 34 \quad 41  \quad 84  \quad 84 \quad 224 $

$1  \quad1 \quad 34 \quad 41  \quad 41  \quad 84 \quad 224 $

$1 \quad 34 \quad 41  \quad84 \quad84 \quad 224 $

$ 1 \quad 41 \quad 84$

 $ 84$ 
\end{center}}
\parbox{3in}{
\begin{center} 
$224$ 

$1 \quad 34 \quad 41 \quad 224$

$1 \quad1 \quad 34 \quad 41   \quad 84  \quad 224 \quad 224 $

$1 \quad 1  \quad 34 \quad 41 \quad 41  \quad 84   \quad 224 \quad 22$

$1 \quad1 \quad 34 \quad 41   \quad 84  \quad 224 \quad 224 $

$1 \quad 34 \quad 41 \quad 224$

 $224$ 
\end{center}}
 
\end{center}

\end{Theorem}

\begin{Theorem}\label{T:Main2}
The Loewy structure of the principal indecomposable modules of $B_1$ are as follows:

\vspace{5pt}
\begin{center}
\parbox{1.5in}
{\begin{center}
$9$

$126 \quad 279$

$9 \quad 9 \quad 36 \quad 90$

$126 \quad 279$

$9$
\end{center}} 
\parbox{1.5in}
{\begin{center}
$36$

$126 \quad 279$

$9 \quad 36\quad 36 \quad 90$

$126 \quad 279$

$36$
\end{center}} 
\parbox{1.5in}
{\begin{center}
$90$

$126 $

$9 \quad 36 \quad 90$

$126$

$90$
\end{center}}
\end{center}

\vspace{15pt}
\begin{center}
\parbox{1.5in}
{\begin{center}
$126$

$9 \quad 36 \quad 90$

$126 \quad 126 \quad 279 $

$9 \quad 36 \quad 90$

$126$
\end{center}} 
\parbox{1.5in}
{\begin{center}
$279$

$9 \quad 36 $

$126 \quad 279 $

$9 \quad 36 $

$279$
\end{center}} 

\end{center}
\end{Theorem}

We note that the Loewy and socle series for the principal indecomposable modules (PIMs) over $F$ are the same, which is also the case for  $A_6, A_7, A_8$ and $A_9$. 

  The Loewy structure of the  PIMs for $FA_6$ is well-known and appears in~\citep{bmodrep}; their module diagrams can be found in ~\citep{bencarl}. The structures for $FA_7$ and $FA_8$ have been calculated in ~\citep{scopes}, while that for $FA_9$ has been done in ~\citep{sie}. The structures for $FA_8$ and $FA_9$ have also been done in characteristic two in ~\citep{ba8} and ~\citep{ba9} respectively. The results for $FA_{10}$ in characteristic $3$ have not been published to our knowledge.

If $A$ is a group algebra over $F$ and $M$ a finitely generated $A$-module, write $L_i(M)$ and $S_i(M)$ for the $i^{th}$ Loewy layer of $M$ and the $i^{th}$ socle layer of $M$, respectively.
We shall write $(M,N)_A$ for $\dim_{F}$Hom$_{A}(M,N)$, $M^{*}$ for Hom$_{F}(M,F)$ regarded as an $A$-module, $(M,N)^{1}_{A}$ for $\dim_{F}$Ext$_{A}^{1}(M,N)$, and $P_M$ for the projective cover of $M$. We also denote a uniserial module of Loewy length 3 with head $S$, heart $T$ and socle $U$ by $\mathfrak{U}(S; T; U)$, and a module $M$ with $L_1(M)= S_1, L_2(M)= S_2 \oplus S_3$ and $L_3(M)= S_4$ by $\mathfrak{D}(S_1; S_2, S_3; S_4)$ (here all the $S_i$ are simple).   

Facts about modular representation theory can be found in  ~\citep{lan}. We use Brauer characters and GAP ~\citep{GAP} to find the composition factors of the $A_{10}$-modules we study.  We will occasionally use a (Benson-Carlson) module diagram to describe an $A$-module $M$. This is a finite directed graph with vertices labelled by simple modules, and with an edge from a vertex $S$ to a vertex $T$ corresponding to a non-trivial element of Ext$_{A}^{1}(S,T)$. The graph must satisfy other additional properties to represent $M$ as described in ~\citep{bencarl}.

In the appendix, we give the decomposition matrix, the Cartan matrix of $A_{10}$ mod 3 (see ~\citep{jamesker} or use \citep{GAP}) and dim$_F$Ext$_{A_{10}}^{1}(S,T)$.
\section{Preliminary results on $A_8$ and $A_9$}\label{A8A9}
\subsection{Results on $A_8$} See ~\citep{scopes}\label{A8}.
\begin{Theorem}\label{T:A8}

The structures of the PIMs for $A_8$ are

\parbox{1.2in}
{\begin{center}
$1$

$13 \quad 35 $

$1 \quad 1 \quad 7 \quad 28 $

$13 \quad 35 $

$1$
\end{center}} \ \ 
\parbox{1in}
{\begin{center}
$7$

$13 \quad 35 $

$1\quad 7 \quad 7 \quad 28$

$13 \quad 35 $

$7$
\end{center}} \ \ 
\parbox{1in}
{\begin{center}
$13$

$1 \quad 7 $

$13 \quad 35 $

$1 \quad 7 $

$13$
\end{center}} \  \
\parbox{1in}
{\begin{center}
$28$

$35 $

$1 \quad 7 \quad 28 $

$35 $

$28$
\end{center}} \  \ 
\parbox{1in}
{\begin{center}
$35$

$1 \quad 28 \quad 7 $

$1 \quad 7 \quad 28 $

$13 \quad 35 \quad 35 $

$1 \quad 7 \quad 28 $

$35$
\end{center}}

\noindent \text{together with}\ 
$\mathfrak{U}(21; 21; 21)$,
$45_1$ and
$45_2$.
\end{Theorem}

\subsection{Results on $A_9$}\label{A9} Siegel's results ~\citep{sie}  recast in our notation are

\begin{Theorem}
\label{Restriction}

~\newline
\begin{center}
$1_{A_9} \downarrow = 1_{A_8}$\ , \
{$7_{A_9} \downarrow = 7_{A_8}$}\ , \ 
 $27_{A_9} \downarrow=\mathfrak{U}(7; 13; 7)$\ , \
$21_{A_9} \downarrow = 21_{A_8}$\ , \
$35_{A_9} \downarrow = 35_{A_8}$\ , \

\vspace{12pt}
$41_{A_9} \downarrow = 28_{A_8} \oplus 13_{A_8}$\ , \
$189_{A_9} \downarrow = 45_1 \oplus 45_2 \oplus \mathfrak{D}(35; 1, 28; 35)$.
\end{center}
\end{Theorem}

\begin{Theorem}
\label{Induction}
~\newline
\begin{center}
$1_{A_8} \uparrow = \mathfrak{U}(1; 7; 1)$\ , \
$7_{A_8} \uparrow = 27 \oplus\mathfrak{D}(7; 1, 21; 7)$\ , \
$13_{A_8} \uparrow = \mathfrak{U}(41; 35; 41)$\ , \
$35_{A_8} \uparrow = 189 \oplus \mathfrak{D}(35; 21, 35; 35)$ \ , \
$28_{A_8} \uparrow = 182 \oplus \mathfrak{D}(41; 1, 7; 41)$\ , \
$21_{A_8} \uparrow = $\parbox{1in}
{\begin{center}
$21$

$7 \quad 35 $

$1 \quad 41 \quad 21$

$7 \quad 35$

$21$
\end{center}}.

\end{center}

\end{Theorem}

\begin{Theorem}\label{T:A9} 
The structures of the PIMs for $A_9$ are

i) The principal $3$-block:
\begin{center}
\parbox{3in}{
\begin{center}
$1$ 

$7 \quad 41 \quad 35 $

$1 \quad 1 \quad 1  \quad 7 \quad 21 \quad 41 \quad 35 \quad 35 $

$1 \quad 1 \quad 7  \quad 7 \quad 7 \quad 21  \quad 41 \quad 41  \quad 35 \quad 35 $

$1 \quad 1 \quad 1  \quad 7 \quad 21 \quad 41 \quad 35 \quad 35 $

$7 \quad 41 \quad 35 $

 $1$ 
\end{center}}

\vspace{15pt}
\parbox{3.2in}{ 
\begin{center} 
 $7$ 

$1 \quad 21 \quad 41 \quad 35 $

$1 \quad 7 \quad 7  \quad 7 \quad 21 \quad 41 \quad 35 \quad 35 $

$1 \quad 1 \quad 1 \quad 7  \quad 7 \quad 21 \quad 21  \quad 41 \quad 41  \quad 35 \quad 35 $

$1 \quad 7 \quad 7  \quad 7 \quad 21 \quad 41 \quad 35 \quad 35 $

$1 \quad 21 \quad 41 \quad 35 $

 $7$ 
\end{center}}
 \parbox{2.8in}{ 
\begin{center} 
  $21$ 

$7 \quad 21 \quad 35 $

$1 \quad 7 \quad 21 \quad 21 \quad 41 \quad 35  $

$1 \quad 7  \quad 7 \quad 21  \quad 41  \quad 35 \quad 35 $

$1 \quad 7 \quad 21 \quad 21 \quad 41 \quad 35  $

$7 \quad 21 \quad 35 $

 $21$ 
 \end{center}}
 
\vspace{20pt}
\parbox{2.3in}{ 
\begin{center} 

$41$ 

$1 \quad 7 \quad 35 $

$1 \quad 7 \quad 21 \quad 41 \quad 41  \quad 35 $

$1 \quad 1  \quad 7  \quad 7 \quad 21  \quad 35 \quad 35 $

$1 \quad 7 \quad 21 \quad 41 \quad 41  \quad 35 $

$1 \quad 7 \quad 35 $

 $41$ 
 
\end{center}}
 \parbox{3.6in}{ 
\begin{center} 

$35$ 

$1 \quad 7 \quad 21 \quad 41 \quad 35 $

$1 \quad 1 \quad 7 \quad 7   \quad 21 \quad 41 \quad 35 \quad 35 \quad 35 $

$1 \quad 1  \quad 7  \quad 7 \quad 21 \quad 21  \quad 41 \quad 41  \quad 35 \quad 35 \quad 35 $

$1 \quad 1 \quad 7 \quad 7   \quad 21 \quad 41 \quad 35 \quad 35 \quad 35 $

$1 \quad 7 \quad 21 \quad 41 \quad 35 $

 $35$ 
\end{center}}
 \end{center}

\vspace{12pt}
ii) Non-principal blocks:

\vspace{12pt}
$\mathfrak{U}(27; 189; 27)$\ , \ $\mathfrak{U}(189; 27; 189)$\ , \ $162$ .

\end{Theorem}

\section{Induction and restriction between $A_9$ and $A_{10}$}

\subsection{Restriction from $A_{10}$ to $A_9$}

\begin{Theorem}
\label{RestrictionA10}
~\newline
\begin{center}
$1_{A_{10}} \downarrow = 1_{A_9}$\ , \
$9_{A_{10}} \downarrow = \mathfrak{U}(1; 7; 1)$\ , \
$34_{A_{10}} \downarrow = 7_{A_9} \oplus 27_{A_9}$\ , \
$36_{A_{10}} \downarrow = \mathfrak{D}(7; 1, 21; 7)$\ , \

\vspace{12pt}
$41_{A_{10}}\downarrow = 41_{A_9} $\ , \
$84_{A_{10}} \downarrow = \mathfrak{D}(21; 7, 35; 21)$\ , \
$90_{A_{10}} \downarrow= \mathfrak{D}(41; 1, 7; 41)$\ , \

\vspace{12pt}
$126_{A_{10}} \downarrow = \mathfrak{D}(35; 21, 35; 35)$\ , \
$224_{A_{10}}\downarrow = 35_{A_9} \oplus 189_{A_9}$\ , \

\vspace{12pt}
$279_{A_{10}} \downarrow = 13_{A_8}\uparrow^{A_9} \oplus 162_{A_9}= \mathfrak{U}(41; 35; 41)$\ , \
$567_{A_{10}} \downarrow = 162_{A_9} \oplus \mathfrak{U}(189; 27; 189)$.
\end{center}

\end{Theorem}
\begin{proof}

This follows easily from Frobenius Reciprocity, block theory and self-duality of all the modules involved. For example, using Brauer characters, $36_{A_{10}}\downarrow$ has composition factors $1_{A_9} +2(7_{A_9}) + 21_{A_9}$. By Frobenius Reciprocity, $L_1(36 \downarrow)= S_1(36 \downarrow)=7$. Hence, self-duality of $36 \downarrow$ gives the result.
\end{proof}

\subsection{Induction from $A_9$ to $A_{10}$}

\begin{Theorem}
\label{InductionA10}
~\newline
\begin{center}
$1_{A_9} \uparrow =1_{A_{10}} \oplus 9_{A_{10}}$\ , \
$7_{A_9} \uparrow =34_{A_{10}} \oplus 36_{A_{10}}$\ , \
$21_{A_9} \uparrow = \mathfrak{D}(84; 1, 41; 84)$\ , \

\vspace{12pt}
$35_{A_9} \uparrow =126_{A_{10}} \oplus 224_{A_{10}}$\ , \
$41_{A_9} \uparrow =41_{A_{10}} \oplus 90_{A_{10}} \oplus 279_{A_{10}}$\ , \

\vspace{12pt}
$162_{A_9} \uparrow = P_{162_{A_9}} \uparrow = P_{162_{A_{10}}} \oplus P_{567_{A_{10}}}$\ ,

\vspace{15pt}
$27_{A_9} \uparrow = $ \parbox{0.5in}
{\begin{center}
$34$

$1 \quad 41 $

$34 \quad 84$

$1 \quad 41$

$34$
\end{center}} 
\ , \
$189_{A_9}\uparrow=567_{A_{10}} \oplus$\parbox{1.5in}
{\begin{center}
$224$

$1\quad 41 \quad 224 $

$1 \quad 34 \quad 84 \quad 224$

$1\quad 41 \quad 224 $

$224$
\end{center}} .

\end{center}
\end{Theorem}

\begin{proof}
With the exception of $27\uparrow$ and $189\uparrow$, all structures follow from block theory, Frobenius Reciprocity, Brauer characters and self-duality.

For $27\uparrow$ and $189\uparrow$, the MeatAxe gives the claimed Loewy layers. 
Furthermore,  the submodule lattice  for $27 \uparrow$ implies that every simple module in $L_i(27\uparrow)$ extends every simple module in $L_{i+1}(27\uparrow)$. 

For $189\uparrow$, we only obtain partial information regarding extensions between modules in $L_3(189\uparrow)$ and modules in both $L_2(189\uparrow)$ and $L_4(189\uparrow)$, namely,

$\bullet \ 34$ extends $1$s, $41$s and $224$s;

$\bullet \ 84$ extends only $1$s and $41$s;

$\bullet \ 224$ extends $1$s and $41$s (and may or may not extend $224$s),

$\bullet \ 1$ in  $L_3(189\uparrow)$ extends $41$s in $L_2(189\uparrow)$ and $L_4(189\uparrow)$ (and possibly other simple modules). 

A partial diagram for $189\uparrow$ is as follows:
 \begin{center}
  \vspace{10pt}
 \parbox{1.5in}
{\begin{center}
$224$

\vspace{10pt}
$1 \quad 41 \quad 224 $

\vspace{10pt}
$224 \quad 34 \quad 84 \quad 1$

\vspace{10pt}
$ 1 \quad 41 \quad 224$

\vspace{10pt}
$224$
\end{center}}.
\put(-78,-5){\line(3,-5){9}}           
\put(-62,-5){\line(5,-6){14}}           
\put(-80,20){\line(3,-2){15}}          
\put(-35,20){\line(-3,-2){18}}     
\put(-63,-5){\line(0,-5){12}}      
\put(-63,-5){\line(-3,-2){18}}       
\put(-85,-5){\line(0, -3){12}}      
\put(-40,-5){\line(-3,-1){40}}       
\put(-40,-5){\line(-3,-2){17}}       
\put(-65,20){\line(0,-5){10}}           
 \end{center}
\end{proof}
\section{Calculation of $(S,T)^{1}_{A_{10}}$ for simple $FA_{10} $-modules $S$ and $T$}\label{ext}
\subsection{Non-principal block}\label{non-prin}
\begin{Lemma}\label{ext9}
\[
(9,S)^1_{A_{10}}=
\begin{cases}
1, &\text{if $S \in\{126, 279\}$;}\\
0, &\text{otherwise.}
\end{cases}
\]
\end{Lemma}

\begin{proof}
We have
\begin{align*}
(9,S)^{1}_{A_{10}} &= (1 \oplus 9, S)^{1}_{A_{10}} &&\text{by block theory,}\\
&=(1\uparrow_{A_9}^{A_{10}}, S)^{1}_{A_{10}}  &&\text{by Theorem~\ref{InductionA10},}\\
&=(1_{A_9},S\downarrow_{A_9}^{A_{10}})^{1}_{A_{9}}  &&\text{by Ext Reciprocity.}
\end{align*}
By Theorem~\ref{T:A8}, Theorem~\ref{Induction} and Theorem~\ref{RestrictionA10},

\[
(1_{A_9},9\downarrow_{A_9}^{A_{10}})^{1}_{A_{9}}
=(1_{A_9}, \mathfrak{U}(1; 7; 1)
)^{1}_{A_9}=(1_{A_9},1\uparrow_{A_8}^{A_9})^{1}_{A_9}=(1_{A_8},1_{A_8})^{1}_{A_8}=0.
\]
By Theorem~\ref{T:A8}, Theorem~\ref{Induction}, Theorem~\ref{RestrictionA10} and Theorem~\ref{InductionA10},

\[
(1_{A_9},36\downarrow_{A_9}^{A_{10}})^{1}_{A_{9}}
=(1_{A_9},36\downarrow_{A_9}^{A_{10}}\oplus 27_{A_9})^{1}_{A_9}=(1_{A_9},7\uparrow_{A_8}^{A_9})^{1}_{A_9}=(1_{A_8},7_{A_8})^{1}_{A_8}=0.
\]

\[
(1_{A_9},90\downarrow_{A_9}^{A_{10}})^{1}_{A_{9}}
=(1_{A_9},90\downarrow_{A_9}^{A_{10}}\oplus 162_{A_9})^{1}_{A_9}=(1_{A_9},28\uparrow_{A_8}^{A_9})^{1}_{A_9}=(1_{A_8},28_{A_8})^{1}_{A_8}=0.
\]

\[
(1_{A_9},126\downarrow_{A_9}^{A_{10}})^{1}_{A_{9}}
=(1_{A_9},126\downarrow_{A_9}^{A_{10}}\oplus 189_{A_9})^{1}_{A_9}=(1_{A_9},35\uparrow_{A_8}^{A_9})^{1}_{A_9}=(1_{A_8},35_{A_8})^{1}_{A_8}=1.
\]

$ \ (1_{A_9},279\downarrow_{A_9}^{A_{10}})^{1}_{A_{9}}
=
(1_{A_9},13\uparrow_{A_8}^{A_9})^{1}_{A_9}=(1_{A_8},13_{A_8})^{1}_{A_8}=1.
$
\end{proof}

A similar proof gives 

\begin{Lemma}\label{ext36}
\[
(36,S)^1_{A_{10}}=
\begin{cases}
1, &\text{if $S \in\{126, 279\}$;}\\
0, &\text{otherwise.}
\end{cases}
\]

\[
(90,S)^1_{A_{10}}=
\begin{cases}
1, &\text{if $S=126 $}\\
0, &\text{otherwise.}
\end{cases}
\]

\[
(126,S)^1_{A_{10}}=
\begin{cases}
1, &\text{if $S \in\{9, 36, 90\}$;}\\
0, &\text{otherwise.}
\end{cases}
\]

\[
(279,S)^1_{A_{10}}=
\begin{cases}
1, &\text{if $S \in\{9, 36\}$;}\\
0, &\text{otherwise.}
\end{cases}
\]
\end{Lemma}

\subsection{Principal block}\label{prin}
\begin{Lemma}\label{ext1}
\[
(1,S)^1_{A_{10}}=
\begin{cases}
1, &\text{if $S \in\{34, 41, 84, 224\}$;}\\
0, &\text{otherwise.}
\end{cases}
\]

\[
(34,S)^1_{A_{10}}=
\begin{cases}
1, &\text{if $S \in\{1, 41, 224\}$;}\\
0, &\text{otherwise.}
\end{cases}
\]

\[
(41,S)^1_{A_{10}}=
\begin{cases}
1, &\text{if $S \in\{1, 34, 84, 224\}$;}\\
0, &\text{otherwise.}
\end{cases}
\]

\[
(84,S)^1_{A_{10}}=
\begin{cases}
1, &\text{if $S \in\{1, 41, 84\}$;}\\
0, &\text{otherwise.}
\end{cases}
\]

\[
(224,S)^1_{A_{10}}=
\begin{cases}
1, &\text{if $S \in\{1, 34, 41, 224\}$;}\\
0, &\text{otherwise.}
\end{cases}
\]

\end{Lemma}

\begin{proof}

The difficulty arises only when computing $(84,S)^1_{A_{10}}$ (noting that, since all the simple $FA_{10}$-modules  are self-dual, $(S,T)^1_{A_{10}} = (T,S)^1_{A_{10}} $). In all other cases, a similar proof to that in the case of non-principal blocks will suffice. 

By block theory, Theorem~\ref{T:A8}, Theorem~\ref{Induction}, Theorem~\ref{RestrictionA10} and Theorem~\ref{InductionA10},
\[
(84,1)^1_{A_{10}} = (84, 1\uparrow_{A_9}^{A_{10}})^1_{A_{10}} =(84\downarrow_{A_9},1_{A_9})^1_{A_9}= ( \mathfrak{D}(21; 7, 35; 21 ), 1)^1_{A_9}.
\]
Furthermore, Theorem~\ref{T:A9} implies that there are only three possibilities:
\begin{center}
\parbox{1in}
{\begin{center}
$21$
 
 \vspace{10pt}
$7 \quad 35 $

\vspace{10pt}
$\ \ 1 \quad21$
\end{center}} 
\ ,  \
\parbox{1in}
{\begin{center}
$21$

\vspace{10pt}
$7 \quad 35 $

\vspace{10pt}
$1 \quad 21$
\end{center}} 
\ , \ 
\parbox{1in}
{\begin{center}
$21$

\vspace{10pt}
$7 \quad 35 $

\vspace{10pt}
$21 \quad 1$
\end{center}} 
\put(-217,20){\line(3,-5){6.5}}      
\put(-217,20){\line(-3,-5){6.5}}     
\put(-227,-3){\line(0,-5){10}}        
\put(-205,-3){\line(0,-5){10}}        
\put(-210,-3){\line(-3,-5){9.5}}      
\put(-222,-3){\line(3,-5){9.5}}        
\put(-40,20){\line(-3,-5){6.5}}      
\put(-130,20){\line(-3,-5){6.5}}     
\put(-35,20){\line(3,-5){6.5}}         
\put(-125,20){\line(3,-5){6.5}}       
\put(-30,-3){\line(-3,-5){6.5}}         
\put(-130,-3){\line(3,-5){7.5}}         
\put(-120,-3){\line(0,-5){10}}         
\put(-138,-3){\line(0,-5){10}}          
\put(-25,-3){\line(0,-5){10}}             
\put(-48,-3){\line(0,-5){10}}              
\end{center}
 However, there is a unique copy of $1$ in $L_3(P_{21_{A_9}})$ and the first structure exists by the submodule lattice of $21\uparrow_{A_8}^{A_9}.f_0$ obtained by the MeatAxe. Therefore, we see that the other two cases cannot occur (see \ref{p84}).

The same proof gives $(84,41)^1_{A_{10}}=1$.

Now using appropriate long exact sequences of cohomology, we have
\[
(84,84)^1_{A_{10}} \geq 1;\quad
(84,34)^1_{A_{10}} \leq 1;\quad
(84,224)^1_{A_{10}} \leq 1;
\]

We will finish determining the dimensions of the above Ext-spaces in Section~\ref{p34p224} and Section~\ref{p84}.
\end{proof}

\section{Induction of 2-step $FA_9$ - modules}\label{2step}
In this section we shall induce to $A_{10}$ each of the non-trivial extensions of a simple module by a simple module for $A_9$.

By Frobenius reciprocity, we have

$\left(
\parbox{0.25in}
{\begin{center}
$1$

$7$
\end{center}}\right)\uparrow^{A_{10}}_{A_9}  .e_0
=$\parbox{0.25in}
{\begin{center}
$1$

$34$
\end{center}}
\ , \
$\left(
\parbox{0.25in}
{\begin{center}
$1$

$41$
\end{center}}\right)\uparrow^{A_{10}}_{A_9}  .e_0
=$\parbox{0.25in}
{\begin{center}
$1$

$41$
\end{center}}
\ , \
$\left(
\parbox{0.25in}
{\begin{center}
$1$

$35$
\end{center}}\right)\uparrow^{A_{10}}_{A_9}  .e_0
=$\parbox{0.25in}
{\begin{center}
$1$

$224$
\end{center}}
\ , \

\vspace{15pt}

$\left(
\parbox{0.25in}
{\begin{center}
$7$

$35$
\end{center}}\right)\uparrow^{A_{10}}_{A_9} .e_0
=$\parbox{0.25in}
{\begin{center}
$34$

$224$
\end{center}}
\ , \
$\left(
\parbox{0.25in}
{\begin{center}
$7$

$41$
\end{center}}\right)\uparrow^{A_{10}}_{A_9}  .e_0
=$\parbox{0.25in}
{\begin{center}
$34$

$41$
\end{center}}
\ , \

\vspace{15pt}

$\left(
\parbox{0.25in}
{\begin{center}
$35$

$35$
\end{center}}\right)\uparrow^{A_{10}}_{A_9}  .e_0
=$\parbox{0.25in}
{\begin{center}
$224$

$224$
\end{center}}
\ , \
$\left(
\parbox{0.25in}
{\begin{center}
$35$

$41$
\end{center}}\right)\uparrow^{A_{10}}_{A_9}  .e_0
=$\parbox{0.25in}
{\begin{center}
$224$

$41$
\end{center}}
\ , \

\vspace{15pt}

$\left(
\parbox{0.25in}
{\begin{center}
$21$

$7$
\end{center}}\right)\uparrow^{A_{10}}_{A_9}  .e_0
=$\parbox{0.5in}
{\begin{center}
$84$

$1 \quad 41$

$84 \quad 34$
\end{center}}
\ , \
$\left(
\parbox{0.25in}
{\begin{center}
$21$

$35$
\end{center}}\right)\uparrow^{A_{10}}_{A_9}  .e_0
=$\parbox{0.5in}
{\begin{center}
$84$

$1 \quad 41$

$84 \quad 224$
\end{center}}
\ . \

\vspace{15pt}
\noindent
 We shall determine the structure of $\left(
\parbox{0.25in}
{\begin{center}
$27$

$189$
\end{center}}\right)\uparrow^{A_{10}}_{A_9}.e_0 $
and
$\left(
\parbox{0.25in}
{\begin{center}
$189$

$27$
\end{center}}\right)\uparrow^{A_{10}}_{A_9}.e_0 $
in Section~\ref{PrinPim}.
\section{Structure of projective modules for $FA_{10}$}\label{pimA10}

\subsection{The non-principal block}\label{NonPrinPim} We will use Section~\ref{non-prin} and self-duality to find the structure of $P_{90}, P_{279}, P_9 $ and $P_{36}$. For $P_{126}$, we will additionally need Landrock's lemma.

The Cartan matrix shows that $P_{90_{A_{10}}}$ has composition factors $9+ 36+3(90) + 2(126) $.
  Also, it has both head and socle isomorphic to $90$ i.e.$\  P_{90_{A_{10}}} =$ \parbox{0.35in}
{\begin{center}
$90$

$X $

$90$
\end{center}} where $X$ has composition factors $9+ 36+ 90 + 2(126) $.
Section~\ref{non-prin} and self-duality shows that $L_1(X)=S_1(X)=126$. Hence, $X=$ \parbox{0.7in}
{\begin{center}
$126$

$9 \quad 36 \quad 90 $

$126$
\end{center}} , and $P_{90_{A_{10}}}$ has Loewy structure as claimed in Theorem~\ref{T:Main2}.

$\bullet \ \text{Similarly}, \ P_{279_{A_{10}}} =$ \parbox{0.35in}
{\begin{center}
$279$

$Y $

$279$
\end{center}} where $Y$ has composition factors $2(9) + 2(36) + 126 + 279$, and  $L_1(Y)=S_1(Y)= 9 \oplus36$. Therefore, the structure of Y can only be one of the following:
\begin{center}
\parbox{0.9in}
{\begin{center}
$S$

$126 \quad 279  \quad T$

$S$ 
\end{center}} $\oplus T$ \quad or \quad
\parbox{0.9in}
{\begin{center}
$9 \quad 36$

$ 126 \quad 279$

$9 \quad 36$

\end{center}}.
\end{center}

Here $\{S, T\} = \{9, 36\}$. But the first case cannot happen since $(S,T)^1_{A_{10}} = 0$ by 
Section~\ref{non-prin}. Thus, the structure of $P_{279_{A_{10}}} $ is as claimed in Theorem~\ref{T:Main2}.

$\bullet$\ We will work out in details the structure of $P_{9_{A_{10}}}$. The structure of  $P_{36_{A_{10}}}$ can be found in a similar way.

The Cartan matrix shows that $P_{9_{A_{10}}}=$\parbox{0.35in}
{\begin{center}
$9$

$Z $

$9$
\end{center}}  where $Z$ has composition factors $2(9) + 36 +90 +2(126) + 2(279)$. It is easy to see that $L_1(Z)=S_1(Z)= 126 \oplus279$. Consequently, the structure of $Z$ can only be one of the following:

\parbox{10in}{
\parbox{1.22in}
{\begin{center}
$Q$

$9 \quad 9  \quad36 \quad 90 \quad R $

$Q$ 
\end{center}}  $\oplus R$ \  or \ 
\parbox{0.8in}
{\begin{center}
$Q$

$9$

$36 \quad 90 \quad R $

$9$

$Q$ 
\end{center}} $\oplus R$ \  or \ 
\parbox{1in}
{\begin{center}
$126 \quad 279$

$9 \quad9 \quad 36 \quad 90 $

$126 \quad 279$

\end{center}}\ or \
\parbox{0.8in}
{\begin{center}
$126 \quad 279$

$9$

$ 36 \quad 90 $

$9$

$126 \quad 279$

\end{center}}.}

Here $\{R, Q\} = \{126, 279\}$. But the first case cannot happen since $(R,Q)^1_{A_{10}} = 0$, while the second and the fourth are impossible since $(9,36)^1_{A_{10}}=0$  by 
Section~\ref{non-prin}. Hence, the structure of $P_{9_{A_{10}}}$ is as claimed in Theorem~\ref{T:Main2}.

$\bullet$\ By Landrock's lemma and the structure of the PIMs obtained so far, 
$P_{126_{A_{10}}}$ has one copy of $9$, one copy of $36$ and one copy of $90$ in $L_2(P_{126_{A_{10}}})$ and $L_4(P_{126_{A_{10}}})$ each. It also has one copy of  $279$ in $L_3(P_{126_{A_{10}}})$. Now self-duality forces $P_{126_{A_{10}}}$ to has the claimed structure in Theorem~\ref{T:Main2}.

 \subsection{The principal block.}\label{PrinPim}
\subsubsection{ The Loewy structure of $P_{34}$ and $P_{224}$}\label{p34p224}
Using Theorem ~\ref{InductionA10}, we have the following filtrations:

$P_{34_{A_{10}}}= P_{27_{A_9}}\uparrow^{A_{10}}.e_0 =\left(
\parbox{0.35in}
{\begin{center}
$27$

$189$

$27$
\end{center}}\right)\uparrow.e_0$ 

$\quad \sim \quad$
\put(40,45){\line(0,-100){100}}
\put(150,25){\line(0,-90){90}}
\parbox{0.45in}
{\begin{center}
$34$

$1 \quad 41$

$34 \quad 84 $

$1 \quad 41$

$34$
\end{center}}
\quad 
\parbox{1.3in}
{\rule{0in}{0.25in}\begin{center}
$224^*$

$1_a \quad 41_a \quad 224_a$

$1_b \quad 224_b \quad34_b \quad 84_b$

$1_c \quad 41_c \quad 224_c$

$224$
\end{center}}
\quad
\parbox{0.5in}{\rule{0in}{0.5in}
\begin{center}
$34_a$

$1_f \quad 41_f$

$34_f \quad 84_f $

$1 \quad 41$

$34$
\end{center}}

\begin{center}
Figure $1$.
\end{center}

$P_{224_{A_{10}}}= P_{189_{A_9}}\uparrow^{A_{10}}.e_0 =\left(
\parbox{0.35in}
{\begin{center}
$189$

$27$

$189$
\end{center}}\right)\uparrow.e_0$ 

$\quad \sim \quad$
\put(100,45){\line(0,-100){100}}
\put(155,25){\line(0,-90){90}}
\parbox{1.25in}
{\begin{center}
$224$

$1\quad 41 \quad 224$

$1 \quad 224 \quad34\quad 84$

$1 \quad 41 \quad 224$

$224_f$
\end{center}}
\quad 
\parbox{0.5in}
{\rule{0in}{0.25in}\begin{center}
$34$

$1_e \quad 41_e$

$34_d\quad 84_d $

$1 \quad 41$

$34_e$
\end{center}}
\quad
\parbox{1.45in}{\rule{0in}{0.5in}
\begin{center}
$224_d$

$1 \quad 41 \quad 224_g$

$1_h \quad 224_h \quad34_g \quad 84_h$

$1_g \quad 41_g \quad 224_e$

$224$
\end{center}}\ .
\begin{center}
Figure $2$.
\end{center}
Figures $1$ and $2$ give $(34,84)^1_{A_{10}}=(224,84)^1_{A_{10}}=0$ as claimed in Lemma~\ref{ext1}.

Section~\ref{prin} implies $224^* \in  L_2(P_{34_{A_{10}}}) $.
Figure $2$ and Section~\ref{prin} show that there is exactly one copy of $34$ in $L_3(P_{224_{A_{10}}})$. Hence, in Figure $1$, $224_a \in L_3(P_{34_{A_{10}}})$ by Landrock's lemma.  Moreover, $34_b$ extends $224_a$ by the structure of $189\uparrow$, and hence it can only be in $L_4(P_{34_{A_{10}}}) , L_5(P_{34_{A_{10}}}) $ or $L_6(P_{34_{A_{10}}}) $. $(34,S)^1_{A_{10}}$ implies it can only be in $L_4(P_{34_{A_{10}}})$.

This, in turn, forces $1_a, 41_a \in L_3(P_{34_{A_{10}}})$ and
\[
1_c, 41_c, 224_c \notin L_6(P_{34_{A_{10}}})
\tag{*}
\]
by the structure of $189\uparrow$ and Section~\ref{prin} respectively.

Now $84_b$ is not in  either $L_5(P_{34_{A_{10}}})$ or $L_6(P_{34_{A_{10}}})$ by Section~\ref{prin} and the Benson-Carlson diagram of $189\uparrow$, so it has to be in $ L_4(P_{34_{A_{10}}})$.
Moreover, Section~\ref{prin} and $189\uparrow$ also imply $224_b \notin L_5(P_{34_{A_{10}}})$. Also, $224_b \notin L_6(P_{34_{A_{10}}})$: otherwise, there exists a submodule $U$  of $\left(
\parbox{0.25in}
{\begin{center}
$27$

$189$
\end{center}}\right)\uparrow^{A_{10}}_{A_9} $
of Loewy length $4$ with simple head $34$ and simple socle $224$. Hence, by duality, there exists a subquotient $U^*$ of $\left(
\parbox{0.25in}
{\begin{center}
$189$

$27$
\end{center}}\right)\uparrow^{A_{10}}_{A_9}$ 
with simple head $224$, simple socle $34$ and Loewy length $4$. This means that, in Figure $2$, $34_d \in L_4(P_{224_{A_{10}}})$, since Section~\ref{prin} forces $224_d\in L_i(P_{224_{A_{10}}})$ for $i\geq 3$ so the diagram of $189\uparrow$ implies $34_g\notin L_4(P_{224_{A_{10}}})$. This in turn forces $1_e,41_e \in L_3(P_{224_{A_{10}}})$. Hence, $84_d \in L_4(P_{224_{A_{10}}})$ since $(224,84)^1_{A_{10}}=0$. Therefore, we have
\[
\text{Loewy length of } \left(
\parbox{0.25in}
{\begin{center}
$189$

$27$
\end{center}}\right)\uparrow^{A_{10}}_{A_9} .e_0
\leq 7 
< 8 
\leq \text{Loewy length of} 
\left(\parbox{0.25in}
{\begin{center}
$27$

$189$
\end{center}}\right)\uparrow^{A_{10}}_{A_9} .e_0 ,
\]
a contradiction. Therefore, $224_b \in L_4(P_{34_{A_{10}}})$.

Next, Landrock's Lemma implies that $34_d \in L_4(P_{224_{A_{10}}})$. The same argument as above shows that Loewy length of $\left(
\parbox{0.25in}
{\begin{center}
$189$

$27$
\end{center}}\right)\uparrow^{A_{10}}_{A_9} .e_0$ is at most $7$, hence $1_b \notin L_6(P_{34_{A_{10}}})$. Moreover, $1_b \notin L_5(P_{34_{A_{10}}})$ by the Benson-Carlson diagram of $189\uparrow.e_0$ and Section~\ref{prin}. Therefore,  $1_b \in L_4(P_{34_{A_{10}}})$. Now (*) gives
\[
\left(\parbox{0.25in}
{\begin{center}
$27$

$189$
\end{center}}\right)\uparrow^{A_{10}}_{A_9} .e_0=
\parbox{1.9in}
{\begin{center}
$34$

$1 \quad 41 \quad 224$

$34 \quad 84 \quad 1 \quad 41 \quad 224$

$1 \quad 41 \quad 1 \quad 224 \quad 34 \quad 84$

$34 \quad 1 \quad 41 \quad 224$

$224$
\end{center}}.
\]
This has Loewy length $6$, so Loewy length of $\left(\parbox{0.25in}
{\begin{center}
$189$

$27$
\end{center}}\right)\uparrow^{A_{10}}_{A_9}.e_0 = 6$. This implies $34_e \in  L_6(P_{224_{A_{10}}})$. The diagram of $27\uparrow$ now gives

\[
\left(\parbox{0.25in}
{\begin{center}
$189$

$27$
\end{center}}\right)\uparrow^{A_{10}}_{A_9}.e_0 =
\parbox{1.8in}
{\begin{center}
$224$

$1 \quad 41 \quad 224 \quad 34$

$1 \quad 224 \quad 34 \quad 84 \quad 1 \quad 41$

$1 \quad 41 \quad 224  \quad 34 \quad 84$

$224 \quad 1 \quad 41$

$34$
\end{center}}.
\]

Now, in Figure $1$, $34_a \notin L_2(P_{34_{A_{10}}})$ by Section~\ref{prin}. In fact, it is in $L_3(P_{34_{A_{10}}})$, otherwise, it must extend (and indeed lie below)  some composition factors in the submodule \parbox{0.45in}
{\begin{center}
$1 \quad 41$

$34$
\end{center}}
of $27\uparrow$, which is impossible.

Also, Section~\ref{prin} implies that $34_f, 84_f \notin L_6(P_{34_{A_{10}}})$; and they are not in $L_7(P_{34_{A_{10}}})$ either, because the diagram of $27\uparrow$ and Section~\ref{prin} imply $1_f, 41_f \notin L_6(P_{34_{A_{10}}})$.
Thus, $34_f, 84_f \in L_5(P_{34_{A_{10}}})$. This, in turns, forces $1_f, 41_f\in L_4(P_{34_{A_{10}}})$. Therefore, we obtain the structure of $P_{34_{A_{10}}}$ as claimed in Theorem~\ref{T:Main1}.

Finally we compute the structure of $P_{224_{A_{10}}}$. By Landrock's lemma, there is one copy of $34$ in $L_5(P_{224_{A_{10}}})$, so $34_g \in L_5(P_{224_{A_{10}}})$. Thus, $224_d \in L_3(P_{224_{A_{10}}})$ by the diagram of $189\uparrow$. Also, since $84$ does not extend $224$, $84_h \in L_4(P_{224_{A_{10}}})$.

 We now have $1_h \in L_5(P_{224_{A_{10}}})$ since, otherwise, the structure of $\left(\parbox{0.25in}
{\begin{center}
$27$

$189$
\end{center}}\right)\uparrow^{A_{10}}_{A_9}.e_0 $
implies that it must extend $224_f$. This means there exists a module V with Loewy length $4$ with unique head $224$ and unique socle $224$ and whose head  $224$ extends $1$ in  $L_2(V)$. By duality again, $1_h$ extends $224_e$, which implies that $224_g$ is extended by $1_h$. Thus, $1_h$ extends both $224_f $ and $224_g$, contradicting Section~\ref{prin}.
A similar argument shows that $224_h \in L_5(P_{224_{A_{10}}})$ and finally the structure of $P_{224_{A_{10}}}$ is as claimed in Theorem~\ref{T:Main1}.
\subsubsection{The Loewy structure of $P_{84}$.}\label{p84} 
We have

\vspace{10pt}
$\left(
\parbox{0.4in}
{\begin{center}
$7$

$1 \quad 21$

$7$
\end{center}}\right)\uparrow^{A_{10}}_{A_9}.e_0 
\ \sim$
\parbox{1.1in}
{\begin{center}
$34$
\parbox{1.1in}{\rule[.1mm]{1in}{0.01in}
\begin{center}
$1 \oplus \parbox{0.4in}
{\begin{center}
$84$

$1 \quad 41$

$84$
\end{center}}$

\end{center}}
\parbox{1.1in}{\rule[.1mm]{1in}{0.01in}
\begin{center}
$34$

\end{center}}
\end{center}}
$=$ 
 \parbox{.9in}
{\begin{center}
$34 \quad 84$

\vspace{10pt}
$1 \quad 41 \quad 1 $

\vspace{10pt}
$34 \quad 84 $
\end{center}}.
\put(-35,-5){\line(3,-5){6}}           
\put(-55,-5){\line(5,-6){9}}           
\put(-27,20){\line(5,-6){9}}         
\put(-50,20){\line(-2,-3){7}}     
\put(-18,-5){\line(-2,-3){7}}      
\put(-40,-5){\line(-2,-3){7}}       
\put(-27,20){\line(-2, -3){7}}      
\put(-45,20){\line(5,-6){9}}       

\noindent
by Frobenius Reciprocity, Section~\ref{ext} and Section~\ref{2step}.

Similarly,

\vspace{10pt}
$\left(
\parbox{0.2in}
{\begin{center}
$41$

$35$

$41$
\end{center}}\right)\uparrow^{A_{10}}_{A_9} .e_0=$
\parbox{0.2in}
{\begin{center}
$41$

$224$

$41$
\end{center}} ,
$\left(
\parbox{0.45in}
{\begin{center}
$35$

$21 \quad 35$

$35$
\end{center}}\right)\uparrow^{A_{10}}_{A_9} .e_0=$
\parbox{0.9in}
{\begin{center}
$224 \quad 84$

$224 \quad1 \quad 41$

$224 \quad 84$
\end{center}},

\vspace{10pt}
$\left(
\parbox{0.15in}
{\begin{center}
$1$

$7$

$1$
\end{center}}\right)\uparrow^{A_{10}}_{A_9} .e_0=$
\parbox{0.3in}
{\begin{center}
$41$

$224$

$41$
\end{center}},
$\left(
\parbox{0.3in}
{\begin{center}
$41$

$1 \quad7$

$41$
\end{center}}\right)\uparrow^{A_{10}}_{A_9} .e_0=$
\parbox{0.45in}
{\begin{center}
$41$

$1 \quad 34$

$41$
\end{center}}.

Furthermore, we have the following filtration  ((3) in ~\citep{sie}):

\vspace{10pt}
$P_{7_{A_{9}}}= P_{28_{A_8}}\uparrow^{A_{9}}.f_0$

$\sim \ $
\parbox{0.37in}
{\begin{center}
$7$

$1 \quad 21 $

$7  $

\end{center}}
\quad
\put(3,45){\line(0,-100){110}}
\parbox{.3in}
{\rule{0in}{0.3in}\begin{center}
$41$

$35$

$41$
\end{center}}
\parbox{.15in}
{\rule{0in}{0.3in}\begin{center}
$\oplus$
\end{center}}
\parbox{.4in}
{\rule{0in}{0.3in}\begin{center}
$35$

$21 \quad35$

$35$
\end{center}}
\quad
\put(8,25){\line(0,-100){90}}
\parbox{0.6in}{\rule{0in}{0.63in}
\begin{center}
$7$

$1 \quad 21$

$7 $
\end{center}} 
\parbox{0.15in}{\rule{0in}{0.63in}
\begin{center}
$\oplus$
\end{center}}
\parbox{0.3in}{\rule{0in}{0.63in}
\begin{center}
$7$

$1 \quad 21$

$7 $
\end{center}}
\parbox{0.2in}{\rule{0in}{0.63in}
\begin{center}
$\oplus$\
\end{center}}
\parbox{0.13in}{\rule{0in}{0.63in}
\begin{center}
$1$

$7$

$1 $
\end{center}}
\parbox{0.2in}{\rule{0in}{0.63in}
\begin{center}
$\oplus$\
\end{center}}
\parbox{0.4in}{\rule{0in}{0.63in}
\begin{center}
$41$

$1 \quad 7$

$41 $
\end{center}}
\quad
\put(0,15){\line(0,-90){80}}
\parbox{0.3in}{\rule{0in}{0.97in}
\begin{center}
$41$

$35$

$41 $
\end{center}} 
\parbox{0.2in}{\rule{0in}{0.97in}
\begin{center}
$\oplus$
\end{center}}
\parbox{0.5in}{\rule{0in}{0.97in}
\begin{center}
$35$

$21 \quad 35$

$35 $
\end{center}} 
\put(0,5){\line(0,-90){70}}
\parbox{0.35in}{\rule{0in}{1.35in}
\begin{center}
$7$

$1 \quad 21$

$7 $
\end{center}}

Finally, the structure of $P_{34_{A_{10}}}$ together with the above data gives the following Loewy structure:

\vspace{10pt}
$P_{34_{A_{10}}}\oplus P_{84_{A_{10}}}= P_{7_{A_9}}\uparrow^{A_{10}}.e_0 $

\vspace{10pt}
$=$
\parbox{4.6in}
{\begin{center}
$34 \quad 84$

$1 \quad 41 \quad 1 \quad 41 \quad 224 \quad 84$

$34 \quad 84 \quad 224 \quad 224 \quad 1 \quad 41 \quad 34 \quad 84 \quad 34 \quad 84 \quad 1 \quad 41$

$41 \quad 224 \quad 84 \quad 1 \quad 41 \quad 1 \quad1 \quad 41 \quad 1 \quad34 \quad 1 \quad 34 \quad 41 \quad 224\quad 84 $

$\quad 34 \quad 84 \quad 34 \quad 84 \quad 1 \quad 1 \quad 224 \quad 224\quad1 \quad 41 \quad34 \quad84$

$41 \quad224 \quad84 \quad1 \quad41 \quad1$

$34 \quad84$
\end{center}}.

\noindent
Thus, the structure of $P_{84_{A_{10}}}$ is as claimed in Theorem~\ref{T:Main1} and $(84,84)^1_{A_{10}}=1$ as claimed in Section~\ref{ext}.

\subsubsection{The Loewy structure of  $P_{41}$.}\label{p41}

Landrock's Lemma, Section~\ref{p34p224}, Section~\ref{p84} and the filtration of $P_{41_{A_9}}\uparrow^{A_{10}}.e_0$ give the positions of the unlabelled simple modules and $1_a$ in the Loewy structure for $P_{41_{A_{10}}}$ as follows:
{\begin{center}
$41$

$1_a \quad 34 \quad 84 \quad 224$

$34\quad 224  \quad84 \quad 1_A \quad 41_A  \quad 41_D \quad 41_E \quad 41_F$

$34 \quad 34 \quad 224 \quad 224 \quad 84 \quad 84 \quad 1_B \quad 41_B \quad 1_D \quad 1_H$

$34 \quad224 \quad 84  \quad 1_ C\quad 41_C \quad41_G \quad 1_G \quad 41_H $

$34 \quad 224 \quad 84 \quad 1_E$

$1_F \quad 41_E$

$41_F$
\end{center}}

\noindent
Since we have accounted for all the $84$s by Landrock's Lemma and the $84$s all come from $\mathfrak{D}(84;1;41;84)$, $1_A, 41_A, 1_B, 41_B, 1_C, 41_C$ are in the respective Loewy layers as above.

Next, from ~\citep{sie}, we see that

\begin{center}
$\left(
\parbox{0.25in}
{\begin{center}
$13$

$1$
\end{center}}\right)\uparrow^{A_{9}}_{A_8}  .f_0
=$\parbox{0.5in}
{\begin{center}
$41$

$35 \quad 1$

$41 \quad 7$

$1$
\end{center}}.
\end{center}
\noindent
Thus,
$M := \left(\left(
\parbox{0.25in}
{\begin{center}
$13$

$1$
\end{center}}\right)\uparrow^{A_{9}}_{A_8}  .f_0\right)\uparrow^{A_{10}}
= $\parbox{0.7in}
{\begin{center}
$41$

$224 \quad 1$

$41_a \quad 34_a$

$1_B$
\end{center}},
since Frobenius Reciprocity  and Lemma~\ref{ext1} implies  $S_1(M)=1$  and $41_a, 34_a\notin L_1(M), L_2(M)$.  
Therefore, we get a copy of $41$ and a copy of 1 in $L_3(P_{41_{A_{10}}})$ that do not come from $21\uparrow_{A_9}^{A_{10}}=\mathfrak{D}(84;1;41;84)$, say $41_D$ and $1_D$ respectively.

Next, Frobenius Reciprocity gives
\begin{center}
$\left(\parbox{0.35in}
{\begin{center}
$41$

$1 \quad 7$

$41$
\end{center}}\right)\uparrow_{A_9}^{A_{10}}.e_0=$
\parbox{0.45in}
{\begin{center}
$41$

$1 \quad 34$

$41$
\end{center}}.
\end{center}

\noindent
so $1_a$ extends a $41_E$ in $L_3(P_{41_{A_{10}}})$, and $41_E \neq 41_A, 41_D$ since $13\uparrow_{A_8}^{A_9}=$
\parbox{0.25in}
{\begin{center}
$41$

$35$

$41_1$
\end{center}} 
and there exists an $FA_9$-module 
\parbox{0.35in}
{\begin{center}
$41$

$1 \quad 7$

$41_2$
\end{center}}. Hence, $41_1\neq 41_2$ because there is only one copy of $35$ in$L_2(P_{41_{A_{9}}})$. Thus, $41_E$ comes from $41_2\uparrow_{A_9}^{A_10}$, $41_D$ comes from $41_1\uparrow_{A_9}^{A_10}$ and $41_A$ comes from $\mathfrak{D}(84;1;41;84)$.

Note that all the factors in $L_2(P_{41_{A_{10}}})$ extend some factors in $L_3(P_{41_{A_{10}}})$ so self-duality implies that the layer immediately above the unique bottom $41$ in $P_{41_{A_{10}}}$ is
\[
1 \quad 34 \quad 84 \quad 224
\]

\noindent
Therefore, the Loewy length of $P_{41_{A_{10}}}$ is $7$ and there is exactly one copy of $1$ in $L_6(P_{41_{A_{10}}})$, say $1_E$. We can also fill in the bottom $41$, say $41_F$.

Now $41_D$ extends $1_D$ so by self-duality and the submodule lattice of  $P_{41_{A_{10}}}$, we must have another copy of $41$ in $L_5(P_{41_{A_{10}}})$, say $41_G$.

Furthermore, since 
\[
M.rad(M)^r\subseteq soc(M)^{n-r} 
\tag{*}
\]

 for any module $M$ with Loewy length $n$ and there is no other copy of $1$ in  $L_2(P_{41_{A_{10}}})$, no other copy of $1$ appears in $L_5(P_{41_{A_{10}}})$. By the Cartan's matrix, we need to account for another copy of $41$ and three other copies of $1$.

Now inducing the first three Loewy layers of $P_{41_{A_9}}$ to $A_{10}$, we get the following filtration:

$W= \left(
\parbox{1.7in}
{\begin{center}
$41$

$1 \quad7 \quad 35$

$1 \quad 7 \quad 21 \quad 41 \quad 41 \quad 35$
\end{center}}\right)\uparrow^{A_{10}}_{A_9} .e_0
\sim$
\parbox{3in}
{\begin{center}
$41$

$1 \quad 34 \quad 224$
\parbox{3in}{\rule[.1mm]{3in}{0.01in}
\begin{center}
$1 \oplus  34  \oplus \parbox{0.6in}
{\begin{center}
$84$

$1\quad 41$

$84$
\end{center}} 
\oplus 41 \oplus 41  \oplus 224$
\end{center}}
\end{center}}

$=$
\parbox{3.3in}
{\begin{center}
$41$

$1 \quad 34 \quad 224 \quad 84$

$1_F \quad 34 \quad 1 \quad 41 \quad 41 \quad 41 \quad 224 $

$84$
\end{center}}.

\noindent
by the first three layers of $P_{41_{A_{10}}}$ obtained so far and that $1_F$ does not extend any modules in $L_3(P_{41_{A_{10}}})$ or $\mathfrak{D}(84;1;41;84)$ in $W$.

Next, by Section~\ref{A8A9}, Theorem ~\ref{InductionA10} and Frobenius Reciprocity, we have

$ X= \left(
\parbox{0.35in}
{\begin{center}
$28$

$ 35$

$1 $
\end{center}}\right)\uparrow_{A_8}^{A_9}.f_0
=
\left(
\parbox{1in}
{\begin{center}
$41$

$35 \quad 1 \quad 7$

$21 \quad 35 \quad 41 \quad 1$

$35 \quad 7$

$1$

\end{center}}
\right)$
and 
$X\uparrow^{A_{10}}.e_0= \parbox{1.5in}
{\begin{center}
$41$

$224 \quad 1 \quad 34 \quad 84$

$224 \quad 41 \quad 1 \quad 1 \quad 41$

$34_a \quad 224_a \quad 84$

$1$

\end{center}},
 $
since $34_a, 224_a \in L_4$. Note that all the $34$s and $224$s in $L_2(P_{41_{A_{10}}})$ and $L_3(P_{41_{A_{10}}})$ have now been accounted for.
Therefore, there is a copy of $1$ in $L_5(P_{41_{A_{10}}})$, say $1_G$.

Similarly,
$ Y= \left(
\parbox{0.35in}
{\begin{center}
$28$

$ 35$

\end{center}}\right)\uparrow_{A_8}^{A_9}.f_0
=
\left(
\parbox{1in}
{\begin{center}
$41$

$35 \quad 1 \quad 7$

$21 \quad 35 \quad 41$

$35 $

\end{center}}
\right)$
\put(-30,-10){\line(-2,-3){7}}  
and 
$Y\uparrow^{A_{10}}.e_0= \parbox{1.5in}
{\begin{center}
$41$

$224 \quad 1 \quad 34 \quad 84$

$41_E \quad 224 \quad 1 \quad 41$

$224 \quad 84$

\end{center}}.$
 \put(-85,-9){\line(3,-5){6}}   

\noindent
Therefore, by self-duality, there is another copy of $41$ in $L_5(P_{41_{A_{10}}})$, say $41_H$.

The remaining copy of $1$ must be in $L_3(P_{41_{A_{10}}})$ or $L_4(P_{41_{A_{10}}})$ by (*) and Section~\ref{ext}. Using the filtration (6) in ~\citep{sie} below and the fact that $P_{41_{A_{10}}}=P_{41_{A_{9}}}\uparrow_{A_9}^{A_{10}}.e_0$, we see that the remaining 1 must be in $L_4(P_{41_{A_{10}}})$ and the structure of  $P_{41_{A_{10}}}$ is as claimed in Theorem~\ref{T:Main1}.

\vspace{10pt}
$P_{41_{A_{9}}}= P_{13_{A_8}}\uparrow^{A_{9}}.f_0$

$\sim \ $
\parbox{0.37in}
{\begin{center}
$41$

$35 $

$41 $

\end{center}}
\quad
\put(3,45){\line(0,-100){110}}
\parbox{.3in}
{\rule{0in}{0.3in}\begin{center}
$1$

$7$

$1$
\end{center}}
\parbox{.15in}
{\rule{0in}{0.3in}\begin{center}
$\oplus$
\end{center}}
\parbox{.4in}
{\rule{0in}{0.3in}\begin{center}
$7$

$1 \quad21$

$7$
\end{center}}
\quad
\put(8,25){\line(0,-100){90}}
\parbox{0.6in}{\rule{0in}{0.63in}
\begin{center}
$41$

$35 $

$41 $
\end{center}} 
\parbox{0.15in}{\rule{0in}{0.63in}
\begin{center}
$\oplus$
\end{center}}
\parbox{0.6in}{\rule{0in}{0.63in}
\begin{center}
$35$

$21 \quad 35$

$35 $
\end{center}}
\put(0,15){\line(0,-90){80}}
\parbox{0.3in}{\rule{0in}{0.97in}
\begin{center}
$1$

$7$

$1 $
\end{center}} 
\parbox{0.2in}{\rule{0in}{0.97in}
\begin{center}
$\oplus$
\end{center}}
\parbox{0.5in}{\rule{0in}{0.97in}
\begin{center}
$7$

$1 \quad 21$

$7 $
\end{center}} 
\put(0,5){\line(0,-90){70}}
\parbox{0.35in}{\rule{0in}{1.35in}
\begin{center}
$41$

$35$

$41$
\end{center}}

\newpage
\subsubsection{The Loewy structure of $P_1$.} Landrock's lemma and the structures of 
$P_{34_{A_{10}}}, P_{41_{A_{10}}}, P_{84_{A_{10}}}$ and $P_{224_{A_{10}}}$ give the positions of the unlabelled simple modules  in the Loewy layers of $P_{1_{A_{10}}}$ as follows:

\vspace{15pt}
{\begin{center}
$1$

$34 \quad 41 \quad 84 \quad 224$

$34 \quad 224 \quad 224   \quad 84 \quad 41 \quad 41 \quad 1_B \quad 1_E \quad 1_F  \quad 1_G$

$ 34 \quad 34 \quad 34 \quad 224 \quad 224 \quad 84 \quad84 \quad 41 \quad 41 \quad 41 \quad 1_C \quad 1^* $

$34 \quad 224 \quad 224 \quad 84 \quad 41 \quad41 \quad 1_D \quad 1^* \quad 1^* \quad 1^* $

$34 \quad84 \quad 224  \quad 41$

$1_A$
\end{center}}

 \noindent
Since $(1,1)^1_{A_{10}}=0$,  $P_{1_{A_10}}$ has Loewy length $7$ and the unique bottom $1$ in  $P_{1_{A_10}}$, say $1_A$ is in the position as claimed above. 

Besides, as we have accounted for all the $84$s by Landrock's lemma and $84$ only comes from $21\uparrow_{A_{9}}^{A_{10}}=\mathfrak{D}(84;1;41;84)$, $1_B, 1_C$ and $1_D$ are as claimed. 

Now, induce the first three Loewy layers of $P_{1_{A_9}}$ to $A_{10}$ to get the following filtration:

 $V= \left(
\parbox{2.1 in}
{\begin{center}
$1$

$34 \quad 41 \quad 224$

$1 \quad 1 \quad 1 \quad 7 \quad 21 \quad 41\quad 35 \quad 35$
\end{center}}\right)\uparrow^{A_{10}}_{A_9} .e_0
\sim$
\parbox{3in}
{\begin{center}
$41$

$34 \quad 41 \quad 224$

\parbox{3in}{\rule[.1mm]{3in}{0.01in}
\begin{center}
$1 \oplus 1 \oplus 1 \oplus \parbox{1in}
{\begin{center}
$84$

$1 \quad 41$

$84$
\end{center}} 
\oplus 41 \oplus 224 \oplus 224$
\end{center}}
\end{center}}

$=$\parbox{3in}
{\begin{center}
$41$

$34 \quad 41 \quad 224 \quad 84$

$1 \quad 1 \quad 1 \quad 1\quad 41  \quad 41\quad 224 \quad 224$

$84$
\end{center}}

\noindent
so there are three other copies of $1$ in $L_3(P_{1_{A_{10}}})$, say $1_E, 1_F$ and $1_G$. 
We observe that  all Ext$^1_{A_{10}}(S,1)$ between $L_2(P_{1_{A_{10}}})$ and $L_3(P_{1_{A_{10}}})$ have been accounted for (here $S \in \{ 34, 41, 224, 84\}$) by Lemma~\ref{ext1}. Hence, no other copies of $1$ are in $L_3(P_{1_{A_{10}}})$.

Next, induce the first four Loewy layers of $P_{1_{A_9}}$ to $A_{10}$, call this induced module $W$. Then we have a filtration for $W$:
\begin{center}
  $W
\sim$
\parbox{4in}
{\begin{center}
$1$

$34 \quad 41 \quad 224 \quad84$

$1 \quad 1 \quad 1 \quad 1 \quad 41 \quad 41 \quad 224 \quad 224$

$84$
\parbox{4in}{\rule[.1mm]{4in}{0.01in}
\begin{center}
$1_H \oplus 1_I \oplus 34  \oplus 34 \oplus 34\oplus \parbox{1in}
{\begin{center}
$84$

$1 \quad 41$

$84$
\end{center}} 
\oplus 41 \oplus 41 \oplus 224 \oplus 224$
\end{center}}
\end{center}}

$=$\parbox{3in}
{\begin{center}
$1$

$34 \quad 41 \quad 224 \quad 84$

$1 \quad 1 \quad 1 \quad 1\quad 41  \quad 41\quad 224 \quad 224 \quad 84$

$1_H\quad 34 \quad 34 \quad 34 \quad 84 \quad 1 \quad 41 \quad 224 \quad 224 \quad 41 \quad 41$

$1_I\quad 84$
\end{center}}
\end{center}
Indeed, since $1_H$ and $1_I$ cannot extend any module in $L_4(W)$ other than $84_b$ and $(84,1)^1_{A_{10}}=1$, at most one of $1_H$ and $1_I$ is in $L_5(W)$. In other words, at least one of $1_H$ and $1_I$ is in $L_4(W)$ (since it cannot appear in $L_1(W), L_2(W)$ or $L_3(W)$). So without loss of generality, we assume $1_H \in L_4(W)$.

Finally, inducing the first five Loewy layers of $P_{1_{A_{9}}}$ to $A_{10}$, we obtain another three copies of $1$ that must be in  $L_4(P_{1_{A_{10}}})$ or  $L_5(P_{1_{A_{10}}})$. Now we have to account for the positions of the four $1^*$. By (*) in Section~\ref{p41}, at least one of the $1^*$ must appear in  $L_4(P_{1_{A_{10}}})$. Using the filtration (2) in \cite{sie} below, and that  $P_{1_{A_{10}}}= P_{1_{A_{9}}}\uparrow^{A_{10}}.e_0$, the remaining three$1^*$ must be in  $L_5(P_{1_{A_{10}}})$ and the structure of $P_{1_{A_{10}}}$ is as claimed in Theorem~\ref{T:Main1}.

\vspace{10pt}
$P_{1_{A_{9}}}= P_{1_{A_8}}\uparrow^{A_{9}}.f_0$

$\sim \ $
\parbox{0.37in}
{\begin{center}
$1$

$7$

$1  $

\end{center}}
\quad
\put(3,45){\line(0,-100){110}}
\parbox{.3in}
{\rule{0in}{0.3in}\begin{center}
$41$

$35$

$41$
\end{center}}
\parbox{.15in}
{\rule{0in}{0.3in}\begin{center}
$\oplus$
\end{center}}
\parbox{.4in}
{\rule{0in}{0.3in}\begin{center}
$35$

$21 \quad35$

$35$
\end{center}}
\quad
\put(8,25){\line(0,-100){90}}
\parbox{0.4in}{\rule{0in}{0.63in}
\begin{center}
$1$

$7$

$1 $
\end{center}} 
\parbox{0.15in}{\rule{0in}{0.63in}
\begin{center}
$\oplus$
\end{center}}
\parbox{0.13in}{\rule{0in}{0.63in}
\begin{center}
$1$

$7$

$1$
\end{center}}
\parbox{0.2in}{\rule{0in}{0.63in}
\begin{center}
$\oplus$\
\end{center}}
\parbox{0.3in}{\rule{0in}{0.63in}
\begin{center}
$7$

$1 \quad 21$

$7 $
\end{center}}
\parbox{0.2in}{\rule{0in}{0.63in}
\begin{center}
$\oplus$\
\end{center}}
\parbox{0.4in}{\rule{0in}{0.63in}
\begin{center}
$41$

$1 \quad 7$

$41 $
\end{center}}
\quad
\put(0,15){\line(0,-90){80}}
\parbox{0.3in}{\rule{0in}{0.97in}
\begin{center}
$41$

$35$

$41 $
\end{center}} 
\parbox{0.2in}{\rule{0in}{0.97in}
\begin{center}
$\oplus$
\end{center}}
\parbox{0.5in}{\rule{0in}{0.97in}
\begin{center}
$35$

$21 \quad 35$

$35 $
\end{center}} 
\put(0,5){\line(0,-90){70}}
\parbox{0.35in}{\rule{0in}{1.35in}
\begin{center}
$1$

$7$

$1 $
\end{center}}

\section{Acknowledgements}
\noindent
We are grateful to Richard Parker for explaining the use of the MeatAxe and thank Juergen Manfred Muller for pointing out the inaccuracies in a previous version of this paper.

\newpage
\appendix
\section*{Appendix}
Decompostion matrix of $A_{10}$  mod 3

$\begin{array}{r|rrrrrrrrrrr} \hline
 & 1
 & 34
 & 41
 & 84
 & 224
 & 9
 & 36
 & 90
 & 126
 & 279
 & 567
 \rule[-7pt]{0pt}{20pt} \\ \hline
1=\chi_{1} & 1 & . & . & . & . & . & . & . & . & . & . \rule[0pt]{0pt}{13pt} \\
35=\chi_{3} & 1 & 1 & . & . & . & . & . & . & . & . & . \\
42=\chi_{5} & 1 & . & 1 & . & . & . & . & . & . & . & . \\
75=\chi_{6} & . & 1 & 1 & . & . & . & . & . & . & . & . \\
84=\chi_{7} & . & . & . & 1 & . & . & . & . & . & . & . \\
160=\chi_{10} & 1 & 1 & 1 & 1 & . & . & . & . & . & . & . \\
210=\chi_{11} & 1 & . & 1 & 2 & . & . & . & . & . & . & . \\
224_{1}=\chi_{12} & . & . & . & . & 1 & . & . & . & . & . & . \\
224_{2}=\chi_{13} & . & . & . & . & 1 & . & . & . & . & . & . \\
300=\chi_{17} & 1 & 1 & 1 & . & 1 & . & . & . & . & . & . \\
350\chi_{19} & 1 & . & 1 & 1 & 1 & . & . & . & . & . & . \\
384_{1}\chi_{20} & 1 & 1 & 1 & 1 & 1 & . & . & . &  & . & . \\
384_{2}\chi_{21} & 1 & 1 & 1 & 1 & 1 & . & . & . & . & . & . \\
525=\chi_{23} & 2 & 1 & 1 & . & 2 & . & . & . & . & . & . \\
9=\chi_{2} & . & . & . & . & . & 1 & . & . & . & . & . \\
36=\chi_{4} & . & . & . & . & . & . & 1 & . & . & . & . \\
90=\chi_{8} & . & . & . & . & . & . & . & 1 & . & . & . \\
126=\chi_{9} & . & . & . & . & . & . & . & . & 1 & . & . \\
225=\chi_{14} & . & . & . & . & . & 1 & . & 1 & 1 & . & . \\
252=\chi_{15} & . & . & . & . & . & . & 1 & 1 & 1 & . & . \\
288=\chi_{16} & . & . & . & . & . & 1 & . & . & . & 1 & . \\
315=\chi_{18} & . & . & . & . & . & . & 1 & . & . & 1 & . \\
450=\chi_{22} & . & . & . & . & . & 1 & 1 & . & 1 & 1 & . \\
567=\chi_{24} & . & . & . & . & . & . & . & . & . & . & 1 \rule[-7pt]{0pt}{5pt} \\
\hline
\end{array}$

\vspace{15pt}
Cartan matrix

\vspace{15pt}
$
\begin{array}{r|rrrrrrrrrrrr} \hline
 & 1
 & 34
 & 41
 & 84
 & 224
 & 9
 & 36
 & 90
 & 126
 & 279
 & 567
 \rule[-7pt]{0pt}{20pt} \\ \hline
1 & 13 & 7 & 9 & 6 & 8 & . & . & . & . & . & . \rule[0pt]{0pt}{13pt} \\
34 & 7 & 7 & 6 & 3 & 5 & . & . & . & . & . & . \\
41 & 9 & 6 & 9 & 6 & 6 & . & . & . & . & . & . \\
84 & 6 & 3 & 6 & 9 & 3 & . & . & . & . & . & . \\
224 & 8 & 5 & 6 & 3 & 10 & . & . & . & . & . & . \\
9 & . & . & . & . & . & 4 & 1 & 1 & 2 & 2 & . \\
36 & . & . & . & . & . & 1 & 4 & 1 & 2 & 2 & . \\
90 & . & . & . & . & . & 1 & 1 & 3 & 2 & . & . \\
126 & . & . & . & . & . & 2 & 2 & 2 & 4 & 1 & . \\
279 & . & . & . & . & . & 2 & 2 & 1 & 3 & . & . \\
567 & . & . & . & . & . & . & . & . & . & . & 1  \rule[-7pt]{0pt}{5pt} \\
\hline
\end{array}
$

dim$_F$Ext$_{A_{10}}^{1}(S,T)$ for $S,T$ simple
\[
\begin{array}{r|rrrrrrrrrrrr} \hline
 & 1
 & 34
 & 41
 & 84
 & 224
 & 9
 & 36
 & 90
 & 126
 & 279
 & 567
 \rule[-7pt]{0pt}{20pt} \\ \hline
1 & . & 1 & 1 & 1 & 1 & . & . & . & . & . & . \rule[0pt]{0pt}{13pt} \\
34 & 1 & . & 1 & . & 1 & . & . & . & . & . & . \\
41 & 1 & 1 & . & 1 & 1 & . & . & . & . & . & . \\
84 & 1 & . & 1 & 1 & . & . & . & . & . & . & . \\
224 & 1 & 1 & 1 & . & 1 & . & . & . & . & . & . \\
9 & . & . & . & . & . & . & . & . & 1 & 1 & . \\
36 & . & . & . & . & . & . & . & . & 1 & 1 & . \\
90 & . & . & . & . & . & . & . & . & 1 & . & . \\
126 & . & . & . & . & . & 1 & 1 & 1 & . & . & . \\
279 & . & . & . & . & . & 1 & 1 & . & . & . & . \\
567 & . & . & . & . & . & . & . & . & . & . & .  \rule[-7pt]{0pt}{5pt} \\
\hline
\end{array}
\]

\bibliography{refs}

\begin{thebibliography}{}

\bibitem[GAP, 2013]{GAP}
 (2013).
\newblock {\em {GAP -- Groups, Algorithms, and Programming, Version 4.7.2},}.
\newblock The GAP Group.

\bibitem[Benson, 1983a]{ba8}
Benson, D. (1983a).
\newblock The Loewy structure of the projective indecomposable modules for
  $A_8$ in characteristic $2$.
\newblock {\em Comm. Algebra}, 11:1395--1432.

\bibitem[Benson, 1983b]{ba9}
Benson, D. (1983b).
\newblock The Loewy structure of the projective indecomposable modules for
  $A_9$ in characteristic $2$.
\newblock {\em Comm. Algebra}, 11:1433--1453.

\bibitem[Benson, 1984]{bmodrep}
Benson, D. (1984).
\newblock {\em Modular Representation Theory: New Trends and Methods}.
\newblock Lecture Notes in Mathematics. Springer-Verlag.

\bibitem[Benson and Carlson, 1987]{bencarl}
Benson, D. and Carlson, J. (1987).
\newblock Diagram methods for modular representations and cohomology.
\newblock {\em Comm. Algebra}, 15:53--121.

\bibitem[James and Kerber, 1981]{jamesker}
James, G. and Kerber, A. (1981).
\newblock {\em The Representation Theory of the Symmetric Group}, volume~16 of
  {\em Encyclopedia of Mathematics and its Applications}.
\newblock Addison-Wesley Publishing Company.

\bibitem[Landrock, 1983]{lan}
Landrock, P. (1983).
\newblock {\em Finite group algebras and their modules}, volume~84 of {\em
  London Mathematical Society Lecture Notes}.
\newblock Cambridge University Press.

\bibitem[Scopes, 1988]{scopes}
Scopes, J. (1988).
\newblock The Loewy structure of the projective indecomposable modules of $A_6,
  A_7$ and $A_8$ in characteristic three.
\newblock {\em Dissertation (Oxford
  University)}.

\bibitem[Siegel, 1991]{sie}
Siegel, S. (1991).
\newblock Projective modules for $A_9$ in characteristic three.
\newblock {\em Comm. Algebra}, 19:3099--3117.

\end{thebibliography}
\bibliographystyle{apalike}
\end{document}